\documentclass[onefignum,onetabnum]{siamonline190516}
\usepackage{todonotes}

\usepackage[english]{babel}

\usepackage{url}            
\usepackage{booktabs}       
\usepackage{amsfonts}       
\usepackage{nicefrac}       
\usepackage{graphicx}
\usepackage{subfig}
\usepackage{amssymb}
\usepackage{amsmath}
\usepackage{float}
\usepackage{bbm}
\usepackage{caption}
\usepackage{xcolor}
\usepackage{epstopdf}
\ifpdf
  \DeclareGraphicsExtensions{.eps,.pdf,.png,.jpg}
\else
  \DeclareGraphicsExtensions{.eps}
\fi

\usepackage{enumitem}
\setlist[enumerate]{leftmargin=.5in}
\setlist[itemize]{leftmargin=.5in}

\renewenvironment{proof}{\textbf{Proof.}}{\QED\bigskip}

\newcommand{\BEAS}{\begin{eqnarray*}}
\newcommand{\EEAS}{\end{eqnarray*}}
\newcommand{\BEA}{\begin{eqnarray}}
\newcommand{\EEA}{\end{eqnarray}}
\newcommand{\BEQ}{\begin{equation}}
\newcommand{\EEQ}{\end{equation}}
\newcommand{\BIT}{\begin{itemize}}
\newcommand{\EIT}{\end{itemize}}
\newcommand{\BNUM}{\begin{enumerate}}
\newcommand{\ENUM}{\end{enumerate}}
\newcommand{\BPM}{\begin{pmatrix}}
\newcommand{\EPM}{\end{pmatrix}}

\newcommand{\BA}{\begin{array}}
\newcommand{\EA}{\end{array}}

\newcommand{\ones}{\mathbf 1}
\newcommand{\con}{\text{con}}
\newcommand{\pen}{\text{pen}}

\newcommand{\reals}{{\mathbb R}}

\newcommand{\Rank}{\mathop{\bf Rank}}

\newcommand{\diag}{\mathop{\bf diag}}

\newcommand{\Co}{{\mathop {\bf Co}}}

\newcommand{\QED}{~~\rule[-1pt]{6pt}{6pt}}

\newsiamremark{remark}{Remark}
\newsiamremark{hypothesis}{Hypothesis}
\crefname{hypothesis}{Hypothesis}{Hypotheses}
\newsiamthm{claim}{Claim}

\usepackage{amsopn}

\definecolor{citecolor}{rgb}{0,0.08,0.45}
\hypersetup{
  colorlinks=true,
  citecolor=citecolor,
  linkcolor=black,
  breaklinks=true
}
\definecolor{changescolor}{rgb}{0.8,0.1,0.65}

\headers{Approximation Bounds for Sparse Programs}{A. Askari, A. d'Aspremont, and L. El Ghaoui}

\title{Approximation Bounds for Sparse Programs\thanks{AdA would like to acknowledge support from the {\em ML and Optimisation} joint research initiative with the {\em fonds AXA pour la recherche} and Kamet Ventures, a Google focused award, as well as funding by the French government under management of Agence Nationale de la Recherche as part of the ``Investissements d'avenir'' program, reference ANR-19-P3IA-0001 (PRAIRIE 3IA Institute). LEG would like to acknowledge support from Berkeley Artificial Intelligence Research (BAIR) and Tsinghua-Berkeley-Shenzhen Institute (TBSI).    The authors would like to thank Zihao Chen and Nilesh Tripuraneni for helpful discussions. 
}}

\newcommand{\specificthanks}[1]{\@fnsymbol{#1}}
\usepackage[symbol]{footmisc}

\author{
  Armin Askari\thanks{UC Berkeley, Berkeley, CA
    (\email{aaskari@berkeley.edu}, \email{elghaoui@berkeley.edu})} \and \hspace{-2mm}\thanks{The Voleon Group}
  \and Alexandre d'Aspremont\thanks{CNRS \& École Normale Supérieure,
    (\email{aspremon@ens.fr}).}
  \and Laurent El Ghaoui \footnotemark[2]
}

\begin{document}


\maketitle

\begin{abstract}
We show that sparsity constrained optimization problems over low dimensional spaces tend to have a small duality gap.
We use the Shapley-Folkman theorem to derive both data-driven bounds on the duality gap, and an efficient primalization procedure to recover feasible points satisfying these bounds. These error bounds are proportional to the rate of growth of the objective with the target cardinality, which means in particular that the relaxation is nearly tight as soon as the target cardinality is large enough so that only uninformative features are added.
\end{abstract}



\section{Introduction}
\label{sec:intro}
We study optimization problems with low rank data and sparsity constraints, written 
\BEQ\label{eq:primal-con}\tag{P-CON}
    p_{\con}(k) \triangleq \min_{\|w\|_0 \leq k} \; f(Xw; y) + \dfrac{\gamma}{2} \|w\|_2^2, 
\EEQ
in the variable $w\in\reals^m$, where $X \in \mathbb{R}^{n \times m}$ is assumed low rank, $y \in \mathbb{R}^n, \gamma > 0$ and  $k\geq 0$.  Here, $\|\cdot\|_0$ stands for the $l_0$-norm (cardinality) of its vector argument. We also study a penalized formulation of this problem written
\BEQ\label{eq:primal-pen}\tag{P-PEN}
    p_{\pen}(\lambda) \triangleq \min_{w} \; f(Xw; y) + \dfrac{\gamma}{2} \|w\|_2^2 + \lambda \|w\|_0 
\EEQ
in the variable $w\in\reals^m$, where $\lambda>0$. We provide explicit upper and lower bounds on $p_\con(k)$ and $p_\pen(\lambda)$ that are a function of the bidual problem and the numerical rank of $X$. We also provide a tractable procedure to compute primal feasible points $w$ that satisfy the aforementioned bounds. We first begin with the case where $f(\cdot)$ is convex and show how to extend the results to the case when $f(\cdot)$ is non-convex.

\paragraph{Related literature}
In a general setting \eqref{eq:primal-con} and \eqref{eq:primal-pen} are NP-hard \cite{Nata95}. A very significant amount of research has been focused on producing tractable approximations and on proving recovery under certain conditions. This is the case in compressed sensing for example, where work stemming from \cite{Dono04,Cand05} shows that $\ell_1$ like penalties recover sparse solutions under various conditions enforcing independence among sparse subsets of variables of cardinality at most $k$. 

The convex quadratic case (i.e. $f(Xw;y) = \|Xw - y\|_2^2 = w^\top Q w + 2y^\top w + y^\top y$ with $X^\top X = Q$) has been heavily studied. \cite{park2018semidefinite} for example relax \eqref{eq:primal-con} to a non-convex quadratically constrained quadratic program (QCQP) for which they invoke the S-procedure to arrive at a convex problem; they also draw a connection between their semidefinite relaxation and a probabilistic interpretation to construct a simple randomized algorithm. In \cite{pilanci2015sparse}, the authors obtain a semidefinite programming (SDP) relaxation of the problem. They also consider the cardinality-penalized version of \eqref{eq:primal-con} and use a convex relaxation that is connected with the reverse Huber penalty. In \cite{soubies2015continous}, the authors compute the biconjugate of the cardinality-penalized objective in one dimension and in the case when $Q$ is identity matrix, and compare the minimum of their problem using a penalty term inspired from the derivation of the biconjugate. In \cite{atamturk2019rank, wei2020convexification, atamturk2018sparse}, the authors take advantage of explicit structure of $Q$ (e.g. when $Q$ is rank one) to arrive at tighter relaxations of \eqref{eq:primal-con} by considering convex hulls of perspective relaxations of the problem. They additionally study the case when there is a quadratic penalty on consecutive observations for smoothness considerations. In \cite{xie2018scalable}, the authors show the equivalence between many of the formulations derived in the above papers and provide scalable algorithms for solving the convex relaxations of \eqref{eq:primal-con}. In \cite{gao2013optimal}, the authors take a different approach by looking at the Lagrangian dual of the problem and decoupling the ellipsoidal level sets by considering separable outer approximations of the quadratic program defining the portfolio selection problem. The non-convex quadratic case has also been studied. Namely, it is a well known fact that a quadratic optimization with one quadratic constraint has zero duality gap and can be solved exactly via SDP even when the quadratic forms are non-convex (see e.g. \cite[Appendix\,B]{Boyd04}).

The Shapley-Folkman theorem, used to construct our bounds, was derived by Shapley and Folkman and first published in \cite{Star69}. In \cite{Aubi76}, the authors used the theorem to derive a priori bounds on the duality gap in separable optimization problems, and showcased applications such as in unit commitment problems. Extreme points of the set of solutions of a convex relaxation are then used to produce good approximations and \cite{Udel16} describes a randomized purification procedure to find such points with probability one.

\paragraph{Contributions}
While the works listed above do produce tractable relaxations of problems \eqref{eq:primal-con} and \eqref{eq:primal-pen} they do not yield a priori guarantees on the quality of these solutions (outside of the sparse recovery results mentioned above) and do not handle the generic low rank case. Our bounds are expressed in terms of the value of the bidual, the desired sparsity level and the rank of $X$, which is often low in practice.

Here, we use the Shapley-Folkman theorem to produce a priori bounds on the duality gap of problems \eqref{eq:primal-con} and \eqref{eq:primal-pen}. Our convex relaxations, which are essentially interval relaxations of a discrete reformulation of the sparsity constraint and penalty, produce both upper and lower approximation bounds on the optima of problems \eqref{eq:primal-con} and \eqref{eq:primal-pen}. These relaxations come with primalization procedures, that is, tractable schemes to construct feasible points satisfying these approximation bounds. Furthermore, these error bounds are proportional to the rate of growth of the objective with the target cardinality $k$, which means, in feature selection problems for instance, that the relaxations are nearly tight as soon as $k$ is large enough so that only uninformative features are added.

\subsection{Notation}
For a vector $u \in \mathbb{R}^m$, let $D(u) = \diag(u_1,\hdots,u_m)$. Let $M^\dagger$ denote the pseudoinverse of the matrix $M$. For a closed function $f(x)$, let $f^\ast(y) \triangleq \max_x x^\top y - f(x)$ denote the convex conjugate and let $f^{\ast \ast}(x)$ be the biconjguate (the conjugate of $f^\ast(x)$). Throughout the paper, we will assume $f$ is closed. If we additionally assume $f$ is convex, then  $f^{**}=f$ (see e.g. \cite[Prop.\,6.1.1]{Hiri96}). For simplicity, we will drop the explicit dependence of $y$ in our objective and simply write $f(Xw)$ instead.

\section{Bounds on the Duality Gap of the Constrained Problem}
\label{sec:l0_cons}

We derive upper and lower bounds on the constrained case \eqref{eq:primal-con} in this section. The penalized case will follow from similar arguments in Section \ref{sec:l0_pen}. In both sections we assume $f(\cdot)$ is convex and show in Section \ref{sec:noncvx} how the results change when $f(\cdot)$ is non-convex. We begin by forming the dual problem. 
\subsection{Dual Problem}
Note that the constrained problem is equivalent to
\begin{align}
    p_\con (k) = \min_{v, u \in \{0,1\}^m} \; f(XD(u) v) + \dfrac{\gamma}{2} v^\top D(u) v \; : \; \ones^\top u \leq k 
\end{align}
in the variables $u \in\reals^m$ and $ u \in \{0,1\}^m$, where $D(u) = \diag (u_1, \hdots, u_m)$, using the fact $D(u)^2 = D(u)$. Rewriting $f(\cdot)$ using its fenchel conjugate and swapping the outer min with the inner max to get a dual, we have $d_\con(k) \leq p_\con (k)$ by weak duality, with 
\begin{align*}
    d_\con(k) = \max_z - f^\ast(z) + \min_{v, u \in \{0,1\}^m} \dfrac{\gamma}{2} v^\top D(u) v + z^\top XD(u) v \; : \; \ones^\top u \leq k 
\end{align*}
in the variable $z\in\reals^n$. Solving the inner minimum over $v$, we have $v^\ast = -\tfrac{1}{\gamma} D(u)^\dagger D(u) X^\top z$. Plugging this back into our problem, we get
\begin{align*}
    d_\con(k) = \max_z - f^\ast(z) + \min_{u \in\{0,1\}^m} -\dfrac{1}{2\gamma} z^\top X D(u) D(u)^\dagger D(u) X^\top z \; : \; \ones^\top u \leq k 
\end{align*} 
Noting that $D(u) D(u)^\dagger D(u) = D(u)$ and that $z^\top X D(u) D(u)^\dagger D(u) X^\top z$ is increasing with $u$, we have
\begin{align*}
    d_\con(k) = \max_{z, \zeta} - f^\ast(z) - \dfrac{1}{2\gamma} s_k(\zeta \circ \zeta) \; : \; \zeta = X^\top z
\end{align*}
where $s_k(\cdot)$ denotes the sum of top $k$ entries of its vector argument (all nonnegative here).

\subsection{Bidual Problem} \label{sec:l0_cons_bd}
Rewriting $s_k(\cdot)$ in variational form, we have
\begin{align*}
    p^{\ast \ast}_\con(k) = d_\con (k) = \max_z \min_{u \in [0,1]^m} - f^\ast(z) -\dfrac{1}{2\gamma} z^\top X D(u) D(u)^\dagger D(u) X^\top z \; : \; \ones^\top u \leq k  
\end{align*}
Note this is equivalent to realizing that the inner minimization in $u$ in the previous section could be computed over the convex hull of the feasible set since the objective is in fact linear in $u$. Using convexity and Sion's minimax theorem we can exchange the inner min and max to arrive at
\begin{align*}
     p^{\ast \ast}_\con(k) = \min_{u \in [0,1]^m}  \max_z - f^\ast(z) -\dfrac{1}{2\gamma} z^\top X D(u) D(u)^\dagger D(u) X^\top z \; : \; \ones^\top u \leq k  
\end{align*}
Since $D(u) D(u)^\dagger D(u) \succeq 0$ for all feasible $u$, we have using conjugacy on the quadratic form 
\begin{align*}
    p^{\ast \ast}_\con(k) = \min_{u \in [0,1]^m}  \max_z \min_v - f^\ast(z) + \dfrac{\gamma}{2} v^\top D(u) v + z^\top XD(u) v \; : \; \ones^\top u \leq k  
\end{align*}
Switching the inner min and max again, using the definition of the biconjugate of $f(\cdot)$ and the relation that $f = f^{\ast \ast}$ since $f(\cdot)$ is closed and convex, we get
\begin{align}\label{eq:bidual-con}
    p^{\ast \ast}_\con(k) = \min_{v, u \in [0,1]^m}  f(XD(u) v) + \dfrac{\gamma}{2} v^\top D(u) v \; : \; \ones^\top u \leq k  \tag{BD-CON}
\end{align}
While \eqref{eq:bidual-con} is non-convex, setting $\tilde{v} = D(u) v$ means it is equivalent to the following convex program 
\BEQ\label{eq:pcon-z}
    p^{\ast \ast}_\con(k) = \min_{\tilde{v}, u \in [0,1]^m}  f(X\tilde{v}) + \dfrac{\gamma}{2} \tilde{v} D(u)^\dagger \tilde{v} \; : \; \ones^\top u \leq k  
\EEQ
in the variables $\tilde v,u\in\reals^m$, where $\tilde{v}^\top D(u)^\dagger \tilde{v}$ is jointly convex in $(\tilde{v}, u)$ since it can be rewritten as a second order cone constraint. To compute $(u^\ast, v^\ast)$, we solve the above problem and set $v^\ast = D(u^\ast)^\dagger \tilde{v}^\ast$. Note also that \eqref{eq:bidual-con} is simply the interval relaxation of the \eqref{eq:primal-con}. In fact, in the analysis that follows, we only rely on $\eqref{eq:bidual-con}$ and not the dual.


\subsection{Duality Gap Bounds and Primalization}
We now derive explicit upper and lower bounds on the optimum of \eqref{eq:primal-con} as a function of the rank of the data matrix $X$ and detail a procedure to compute a primal feasible solution that satisfies the bounds. An equivalent analysis will follow for the penalized case.

\begin{theorem}\label{prop:gap}
Suppose $X = U_r \Sigma_r V_r^\top$ is a compact, rank-$r$ SVD decomposition of $X$. From a solution $(v^\ast, u^\ast)$ of~\eqref{eq:bidual-con} with objective value $t^\ast$, with probability one, we can construct a point with at most $k + r + 2$ nonzero coefficients and objective value OPT satisfying
\BEQ\label{eq:gap}
p_\con(k + r + 2)  \leq  OPT \leq p^{\ast\ast}_\con(k)\leq p_\con(k) \tag{Gap-Bound}
\EEQ
by solving a linear program written
\BEQ\label{eq:lp-sf}
\BA{ll}
\mbox{minimize} &c^\top u \\
\text{subject to} &f(U_r z^\ast) + \sum_{i=1}^m u_i \tfrac{\gamma}{2} v_i^{\ast^2} = t^\ast \\
& \sum_{i=1}^m u_i \leq k \\
& \sum_{i=1}^m u_i \ell_i v_i^\ast = z^\ast\\
& u \in [0,1]^m
\EA\EEQ
in the variable $u\in\reals^m$ where $c \sim \mathcal{N}(0, I_m)$, $z^\ast=\Sigma_r V_r^\top D(u^\ast) v^\ast$.
\end{theorem}
\begin{proof}
Making the variable substitution $\Sigma_r V_r^\top D(u) v = z$, \eqref{eq:bidual-con} can be rewritten as
\begin{align*}
    p^{\ast\ast}_\con (k) = \min_{v, u \in [0,1]^m} \; f(U_r z) +\dfrac{\gamma}{2}  v^\top D(u) v \;\; : \;\; \ones^\top u \leq k, \; \Sigma_r V_r^\top D(u) v = z\; 
\end{align*}
and in epigraph form as
\[\BA{ll}
\mbox{minimize} & t \\
\text{subject to} & \begin{bmatrix}
t  \\
k \\
z 
\end{bmatrix} \in \begin{bmatrix}
f(U_r z) \\
\mathbb{R}^+ \\
0
\end{bmatrix} + \sum_{i=1}^m u_i \begin{bmatrix}
\tfrac{\gamma}{2} v_i^2 \\
1 \\
\ell_i v_i
\end{bmatrix}\\
& u \in[0,1]^m
\EA\]
in the variables $t\in\reals$, $z\in\reals^n$ and $v,u\in\reals^m$ where $\ell_i$ is the $i^\mathrm{th}$ column of $\Sigma_r V_r^\top$. Note the above is equivalent to
\[\BA{ll}
\mbox{minimize} & t \\
\text{subject to} & \begin{bmatrix}
t  \\
k \\
z 
\end{bmatrix} \in \begin{bmatrix}
f(U_r z) \\
\mathbb{R}^+ \\
0
\end{bmatrix} + \sum_{i=1}^m \Co \left\{ 0, \begin{bmatrix}
\tfrac{\gamma}{2} v_i^2 \\
1 \\
\ell_i v_i
\end{bmatrix} \right\}
\EA\]
in the variables $t\in\reals$, $z\in\reals^n$ and $v\in\reals^m$. The Shapley Folkman Theorem \cite{Star69} shows that for any
\begin{align*}
x \in \sum_{i=1}^m \Co \left\{ 0, \begin{bmatrix}
\tfrac{\gamma}{2} v_i^2 \\
1 \\
\ell_i v_i
\end{bmatrix} \right\}
\end{align*}
there exists some $\bar{u} \in [0,1]^m$ such that
\begin{align*}
x = \sum_{i \in S} \bar{u}_i \begin{bmatrix}
\tfrac{\gamma}{2} v_i^2 \\
1 \\
\ell_i v_i
\end{bmatrix}  + \sum_{i \in S^c} \bar{u}_i \begin{bmatrix}
\tfrac{\gamma}{2} v_i^2 \\
1 \\
\ell_i v_i
\end{bmatrix}
\end{align*}
where $S = \{i \; | \; \bar{u}_i \not = \{0,1\}\}$ and $|S| \leq r + 2$. Let $(t^\ast, z^\ast, v^\ast, u^\ast)$ be optimal for \eqref{eq:bidual-con}. Then there exists $s_1 \geq 0$ such that
\begin{align*}
\begin{bmatrix}
t^\ast  \\
k -s_1\\
z^\ast 
\end{bmatrix} = \begin{bmatrix}
f(U_r z^\ast) \\
0\\
0
\end{bmatrix} + \sum_{i=1}^m \Co \left\{ 0, \begin{bmatrix}
\tfrac{\gamma}{2} v_i^{\ast^2} \\
1 \\
\ell_i v_i^\ast
\end{bmatrix} \right\}
\end{align*}
From above, we know there exists $\bar{u}_i$ that satisfies these equality constraints, with at most $r + 2$ non-binary entries. In fact, we can compute this $\bar{u}$ by solving a linear program. To see this, given optimal $(t^\ast,z^\ast, v^\ast, u^\ast)$ for the epigraph reformulation of \eqref{eq:bidual-con}, consider the following linear program
\BEQ
\BA{ll}
\mbox{minimize} &c^\top u \\
\text{s.t} \;\;&f(U_r z^\ast) + \sum_{i=1}^m u_i \tfrac{\gamma}{2} v_i^{\ast^2} = t^\ast \\
& \sum_{i=1}^m u_i \leq k \\
& \sum_{i=1}^m u_i \ell_i v_i^\ast = z^\ast\\
& u \in [0,1]^m
\EA\EEQ
in the variable $u\in\reals^m$, where $c \sim \mathcal{N}(0, I_m)$. The problem is feasible since $u^\ast$ is feasible. This is a linear program with $2m + r + 2$ constraints, of which $m$ will be saturated at a non-degenerate basic feasible solution. This implies that at least $m-r-2$ constraints in $0 \leq u \leq 1$ are saturated with probability one, so at least $m-r-2$ coefficients of $u_i$ will be binary at the optimum.

Now, we primalize as follows: given $(t^\ast, z^\ast, v^\ast, \bar{u})$ where $\bar u$ is a non-degenerate basic feasible solution of the LP in~\eqref{eq:lp-sf}, let $S = \{ i\; | \; \bar{u}_i \not \in \{0,1\}\}$ and define
\begin{align*}
\begin{cases}
\tilde{v}_i = \bar{u}_i v_i^\ast, \;\; \tilde{u}_i = 1 & \text{$i \in S$} \\
\tilde{v}_i = v_i^\ast, \;\; \tilde{u}_i = \bar{u}_i & \text{$i \in S^c$}
\end{cases}
\end{align*}
We now claim that $(z^\ast, \tilde{v}, \tilde{u})$ is feasible for the primal problem $p_\con(k + r + 2)$ and has objective value smaller than $p^{\ast\ast}_\con(k)$. By construction, $\tilde{u} \in \{0,1\}^m$ and $\ones^\top \tilde{u} = \|\tilde{u}\|_0 \leq k + r + 2$. Furthermore, we have
\begin{align*}
z^\ast &= \sum_{i=1}^m \bar{u}_i \ell_i v_i^\ast \\
&= \sum_{i \in S} \bar{u}_i \ell_i v_i^\ast  + \sum_{i \in S^c} \bar{u}_i \ell_i v_i^\ast \\
&= \sum_{i \in S} \tilde{u}_i \ell_i \tilde{v_i} + \sum_{i \in S^c} \tilde{u}_i \ell_i \tilde{v_i}
\end{align*}
hence $(z^\ast, \tilde v, \tilde{u})$ is feasible for $p_\con(k + r + 2)$ in~\eqref{eq:pcon-z} and reaches an objective value OPT satisfying
\BEAS
t^\ast & = & f(U_r z^\ast) + \tfrac{\gamma}{2}\left(\sum_{i \in S}^m \bar{u}_i v_i^{\ast 2} + \sum_{i \in S^c}^m \bar{u}_i v_i^{\ast 2}\right) \\
& \geq & f(U_r z^\ast) + \tfrac{\gamma}{2}\left(\sum_{i \in S}^m \bar{u}_i^2  v_i^{\ast 2} + \sum_{i \in S^c}^m \bar{u}_i v_i^{\ast 2}\right) \\
& = & f(U_r z^\ast) + \tfrac{\gamma}{2}\left(\sum_{i \in S}^m \tilde{u}_i \tilde{v_i}^2+ \sum_{i \in S^c}^m \tilde{u}_i \tilde v_i^{2}\right) \\
& \equiv & OPT
\EEAS
Since $(z^\ast, \tilde{v}, \tilde{u})$ is feasible for $p_\con(k+r+2)$ we have $p_\con (k+r+2) \leq OPT$ and the result follows.
\end{proof}

This means that the primalization procedure will always reconstruct a point with at most $k + r + 2$ nonzero coefficients, with objective value at most $p_\con(k) - p_\con(k + r + 2)$ away from the optimal value $p_\con(k)$. Note that this bound does not depend on the value of $\gamma>0$ which could be arbitrarily small and could simply be treated as a technical regularization term.

\section{Bounds on the Duality Gap of the Penalized Problem}\label{sec:l0_pen}
The analysis for the penalized case is very similar to that of the constrained case. We start with deriving the dual problem.

\subsection{Dual Problem}
The penalized problem is equivalent to
\begin{align}
    p_\pen (\lambda) = \min_{v, u\in \{0,1\}^m} f(XD(u) v) + \dfrac{\gamma}{2}v^\top D(u) v + \lambda \ones^\top u 
\end{align}
in the variables $u,v \in\reals^m$. Rewriting $f$ using its fenchel conjugate, switching the min and max, and solving the minimization over $v$ we have
\begin{align*}
    d_\pen (\lambda) = \max_z -f^\ast(z) + \min_{u\in\{0,1\}^m} - \dfrac{1}{2\gamma} z^\top XD(u) D(u)^\dagger D(u) X^\top z + \lambda \ones^\top u,
\end{align*}
Using
\begin{align*}
    \min_{u\in\{0,1\}^m} - \dfrac{1}{2\gamma} z^\top XD(u) D(u)^\dagger D(u) X^\top z + \lambda \ones^\top u = \sum_{i=1}^m \min\Big(0, \lambda - \tfrac{1}{2\gamma} (X^\top z)_i^2\Big)
\end{align*}
the dual problem then becomes
\begin{align*}
    d_\pen(\lambda) = \max_z -f^\ast(z) +\sum_{i=1}^m \min\Big(0, \lambda - \tfrac{1}{2\gamma} (X^\top z)_i^2\Big)
\end{align*}
with $d_\pen(\lambda) \leq p_\pen(\lambda)$.
\subsubsection{Bidual}
Rewriting the second term of our objective in variational form we have
\begin{align*}
    p^{\ast \ast}_\pen (\lambda) = d^\ast(\lambda) = \max_z \min_{u \in [0,1]^m} -f^\ast(z)  - \dfrac{1}{2\gamma} z^\top XD(u) D(u)^\dagger D(u) X^\top z + \lambda \ones^\top u 
\end{align*}
Performing the same analysis as for the constrained case (c.f. Section \ref{sec:l0_cons_bd}), we get
\begin{align}\label{eq:bidual-pen}
    p^{\ast \ast}_\pen (\lambda) = \min_{v,u \in [0,1]^m} f(XD(u) v) + \dfrac{\gamma}{2} v^\top D(u) v + \lambda \ones^\top u \tag{BD-PEN}
\end{align}
in the variables $u,v \in\reals^m$, which can be recast as a convex program as above.

\begin{corollary}\label{prop:pen-gap}
Suppose $X = U_r \Sigma_r V_r^\top$ is a compact, rank-$r$ SVD decomposition of $X$. From a solution $(v^\ast, u^\ast)$ of~\eqref{eq:bidual-pen} with objective value $t^\ast$, with probability one, we can construct a point with objective value OPT satisfying
\BEQ\label{eq:pen-gap}
     p^{\ast\ast}_\pen(\lambda) \leq p_\pen (\lambda) \leq OPT \leq p^{\ast\ast}_\pen (\lambda) + \lambda (r + 1) \tag{Gap-Bound-Pen}
\EEQ
by solving a linear program written
\BEQ\label{eq:lp-sf-pen}
\BA{ll}
\mbox{minimize} & c^\top u \\
\text{s.t} \;\;& f(U_r z^\ast) + \sum_{i=1}^m u_i \tfrac{\gamma}{2} v_i^{\ast^2} + \lambda u_i = t^\ast \\
& \sum_{i=1}^m u_i \ell_i v_i^\ast = z^\ast\\
& u \in [0,1]^m
\EA\EEQ
in the variable $u\in\reals^m$ where $c \sim \mathcal{N}(0, I_m)$ and $z^\ast=\Sigma_r V_r^\top D(u^\ast) v^\ast$.
\end{corollary}
\begin{proof}
The primalization procedure is analogous to the constrained case, the only difference being the linear program becoming~\eqref{eq:lp-sf-pen}. We then get the chain of inequalities in~\eqref{eq:pen-gap} which means that starting from an optimal point of~\eqref{eq:bidual-pen} the primalization procedure will generate a feasible point with objective value at most $\lambda (r + 1)$ larger than that of the original problem~\eqref{eq:primal-pen}.
\end{proof}

\subsection{Connections with other Relaxations}
We first draw the connection between the penalty term in the bidual and the reverse Huber penalty. The reverse Huber function is defined as
\begin{align*}
    B(\zeta) &= \dfrac{1}{2} \min_{0 \leq \nu \leq 1} \nu + \dfrac{\zeta^2}{\nu} \\
    &= \begin{cases}
    |\zeta| & \text{if } |\zeta| \leq 1 \\
    \dfrac{\zeta^2 + 1}{2} & \text{o.w}
    \end{cases}
\end{align*}
We have
\begin{align*}
    \min_{\substack{u \in [0,1]^m \\ \ones^\top u \leq k}} x^\top D(u)^{-1} x = \max_{t > 0} \sum_{i=1}^n t B\Big( \dfrac{|x_i|}{\sqrt{t}}\Big) - \dfrac{1}{2}tk
\end{align*}

There is a direct connection between the second representation of \eqref{eq:bidual-con} (based on the variable substitution $\tilde{v} = D(u) v$) and the well-known perspective based relaxation \cite{frangioni2006perspective} (a similar argument can also be made for \eqref{eq:bidual-pen}). Note that \eqref{eq:primal-con} is equivalent to
\begin{align*}
    \min_{x,u,v} f(x) + \ones^\top v \; : \; u \in \{0,1\}^m, \; \ones^\top u \leq k, \; u_i v_i \geq x_i^2,\; i = 1,\hdots,m
\end{align*}
To see this, assume that $x$ is optimal for \eqref{eq:primal-con}. If $u$ encodes the sparsity pattern of $x$, we simply set $v_i = x_i^2$ so we have $\ones^\top v = x^\top x$ and that triplet $(x,u,v)$ is feasible for the above problem. Similarly, if $(x,u,v)$ are optimal for the above representation, then $x_i = 0$ if $u_i = 0$ and $x_i^2 = v_i$ otherwise. Similarly, $\ones^\top v = x^\top x$ and $x$ is feasible for \eqref{eq:primal-con}. Relaxing $u \in \{u \; | \; u \in [0,1]^m, \; \ones^\top u \leq k\}$ and replacing $f^{\ast \ast}$ with $f$ results in the perspective relaxation of the problem which is equivalent to \eqref{eq:bidual-con}.

\subsection{Extension to Non-Convex Setting} \label{sec:noncvx}
The gap bounds derived above can be extended to the case when $f$ is non-convex. Starting from \eqref{eq:primal-con} and following the structure of \eqref{eq:bidual-con}, consider the relaxation
\begin{align*}
    p^{\ast \ast}_\con(k) = \min_{v, u \in [0,1]^m}  f^{\ast \ast}(XD(u) v) + \dfrac{\gamma}{2} v^\top D(u) v \; : \; \ones^\top u \leq k  
\end{align*}
where $f(\cdot)$ in~\eqref{eq:bidual-con} has been replaced by its convex envelope $f^{\ast \ast}(\cdot)$ (i.e. the largest convex lower bound on $f$). By construction, this constitutes a lower bound on \eqref{eq:primal-con}. The analysis follows the same steps as in the proof of Theorem \ref{prop:gap}, replacing $f$ with $f^{\ast \ast}$ everywhere. The only bound that changes is $p_\con(k+r+2) \leq OPT$ since the objective defining $OPT$ uses $f^{\ast \ast}$ while that defining $p_\con(k+r+2)$ uses $f$. For a non-convex function, we can define the lack of convexity $\rho(f) = \sup_w f(Xw) - f^{\ast \ast}(Xw)$ with $\rho(f) \geq 0$. We then have $-\rho(f) \leq f^{\ast \ast}(U_r z^\ast) - f(U_r z^\ast)$ and then chain of inequalities in~\eqref{eq:gap} becomes
\begin{align*}
    p_\con(k + r + 2) - \rho(f) \leq  OPT \leq p^{\ast\ast}_\con(k)\leq p_\con(k)
\end{align*}
The exact same analysis and reasoning can be applied to the penalized case to arrive at
\begin{align*}
    p^{\ast\ast}_\pen(\lambda) - \rho(f) \leq p_\pen (\lambda) - \rho(f) \leq OPT \leq p^{\ast\ast}_\pen (\lambda) + \lambda (r + 1)
\end{align*}

    

\section{Quadratically Constrained Sparse Problems}
\label{sec:relaxation_cons}

In this section, we consider a version of \eqref{eq:primal-con} where the $\ell_2$ penalty is replaced by a hard constraint. The explicit $\ell_2$ constraint proves useful to get tractable bounds when solving approximate versions of~\eqref{eq:primal-con} where $X$ has low numerical rank (see Section \ref{sec:general}). We follow the same analysis as before and derive similar duality gap bounds and primalization procedures. We omit some steps of the analysis for brevity and refer the reader to Sections \ref{sec:l0_cons} and \ref{sec:l0_pen} for more details. We assume $f$ is convex and can extend the analysis to the non-convex setting using the same arguments in Section \ref{sec:noncvx} (for brevity we omit this). We wish to point out that there is nothing enlightening about the proofs in this section and on a first pass the reader can skip directly to Section \ref{sec:general}.

\subsection{$\ell_2$--$\ell_0$ Constrained Optimization} 
As before, we first derive dual and bidual problems in the quadratically constrained case.
\subsubsection{Dual}
Note that the $\ell_2$-constrained problem is equivalent to
\begin{align*}
    p^\ast_\con (k) = \min_{v, u \in \{0,1\}^m} \; f(XD(u) v) \; : \; \ones^\top u \leq k, \; v^\top D(u) v \leq \gamma
\end{align*}
where $D(u) = \diag (u_1, \hdots, u_m)$ and we use the fact $D(u)^2 = D(u)$.  Rewriting $f$ using its fenchel conjugate, introducing a dual variable $\eta$ for the $\ell_2$ constraint, swapping the outer min with the inner max via weak duality, and solving the minimum over $v$ we have
\begin{align*}
    d^\ast_\con(k) = \max_{z, \eta \geq 0} - f^\ast(z) - \dfrac{\eta \gamma}{2} + \min_{u \in\{0,1\}^m} -\dfrac{1}{2\eta} z^\top X D(u) D(u)^\dagger D(u) X^\top z \; : \; \ones^\top u \leq k 
\end{align*} 
where $d^\ast_\con(k) \leq p^\ast_\con (k)$. This further reduces to
\begin{align*}
    d^\ast_\con(k) = \max_{z, \eta \geq 0} - f^\ast(z) - \dfrac{\eta \gamma}{2} - \dfrac{1}{2\eta} s_k(\zeta \circ \zeta) \; : \; \zeta = X^\top z
\end{align*}
where $s_k(\cdot)$ denotes the sum of top $k$ entries of its vector argument. Note the problem is convex since the latter term is the perspective function of $s_k(\zeta \circ \zeta)$.

\subsubsection{Bidual}
Rewriting $s_k(\cdot)$ in variational form, we have that
\begin{align*}
    p^{\ast \ast}_\con(k) = d^\ast_\con (k) = \max_{z, \eta \geq 0} \min_{u \in [0,1]^m} - f^\ast(z) - \dfrac{\eta \gamma}{2} -\dfrac{1}{2\eta} z^\top X D(u) D(u)^\dagger D(u) X^\top z \; : \; \ones^\top u \leq k  
\end{align*}
Swapping the min and max, and using the Fenchel conjugate of the quadratic form we have
\begin{align*}
     p^{\ast \ast}_\con(k) = \min_{u \in [0,1]^m}  \max_{z, \eta \geq 0} \min_v - f^\ast(z) - \dfrac{\eta \gamma}{2} + \dfrac{\eta}{2} v^\top D(u) v + z^\top X D(u) v \; : \; \ones^\top u \leq k  
\end{align*}
Switching the inner min and max again and using the definition of the biconjugate conjugate of $f(\cdot)$ and computing the maximum over $\eta$, we arrive at
\BEQ\label{eq:bidual-q}
    p^{\ast \ast}_\con(k) = \min_{v, u \in [0,1]^m}  f(XD(u) v) \; : \; \ones^\top u \leq k, \;\; v^\top D(u) v \leq \gamma
\EEQ
which can be rewritten as a convex program (c.f. Section \ref{sec:l0_cons_bd}).

\begin{corollary}\label{prop:gap-l2}
Suppose $X = U_r \Sigma_r V_r^\top$ is a compact, rank-$r$ SVD decomposition of $X$. From a solution $(v^\ast, u^\ast)$ of~\eqref{eq:bidual-q} with objective value $t^\ast$, with probability one, we can construct a point with objective value OPT satisfying
\BEQ\label{eq:gap-l2}
     p^\ast_\con(k + r + 2) - \rho(f) \leq  OPT \leq p^{\ast\ast}_\con(k)\leq p^\ast_\con(k) \tag{Gap-Bound2}
\EEQ
by solving a linear program written
\BEQ\label{eq:lp-sf-pen-q}
\BA{ll}
\mbox{minimize} & c^\top u \\
\text{subject to} \;\; 
&\sum_{i=1}^m u_i \leq k \\
& \sum_{i=1}^m u_i  v_i^{\ast^2}  \leq \gamma \\
& \sum_{i=1}^m u_i \ell_i v_i^\ast = z^\ast\\
& u \in [0,1]^m
\EA\EEQ
in the variable $u\in\reals^m$, where $c \sim \mathcal{N}(0, I_m)$ and $(t^\ast, v^\ast)$ are optimal for the bidual, with $z^\ast=\Sigma_r V_r^\top D(u^\ast) v^\ast $.
\end{corollary}
\begin{proof}
Following the analysis in Section \ref{sec:l0_cons}, let $X = U_r \Sigma_r V_r^\top$ be a compact, rank-$r$ SVD decomposition of $X$. Making the variable substitution $\Sigma_r V_r^\top D(u) v = z$, our bidual can be rewritten in epigraph form as
\[\BA{ll}
\mbox{minimize} & t \\
\text{subject to} & \begin{bmatrix}
t  \\
k \\
\gamma \\
z 
\end{bmatrix} 
\in \begin{bmatrix}
f(U_r z) \\
\mathbb{R}^+ \\
\mathbb{R}^+ \\
0
\end{bmatrix} + \sum_{i=1}^m \Co \left\{ 0, \begin{bmatrix}
0 \\
1 \\
v_i^2 \\
\ell_i v_i
\end{bmatrix} \right\}\\
& u \in [0,1]^m
\EA\]
in the variables $t\in\reals$, $z\in\reals^n$ and $v,u\in\reals^m$, where $\ell_i$ is the $i$th column of $\Sigma_r V_r^\top$. Note that From the Shapley Folkman lemma \cite{Star69}, there exists some $\bar{u} \in [0,1]^m$ such that
\begin{align*}
x = \sum_{i \in S} \bar{u}_i \begin{bmatrix}
0 \\
1 \\
v_i^2 \\
\ell_i v_i
\end{bmatrix}  + \sum_{i \in S^c} \bar{u}_i \begin{bmatrix}
0 \\
1 \\
v_i^2 \\
\ell_i v_i
\end{bmatrix}
\end{align*}
where $S = \{i \; | \; \bar{u}_i \not = \{0,1\}\}$ and $|S| \leq r + 2$ (note we disregard the first entry of the vector and hence it ir $r+2$ and not $r+3$). Now, let $(t^\ast, z^\ast, v^\ast, u^\ast)$ be optimal for the bidual. That means, there exists $s_1, s_2 \geq 0$ such that
\begin{align*}
\begin{bmatrix}
t^\ast  \\
k -s_1\\
\gamma - s_2 \\
z^\ast 
\end{bmatrix} = \begin{bmatrix}
f(U_r z^\ast) \\
0\\
0
\end{bmatrix} + \sum_{i=1}^m \Co \left\{ 0, \begin{bmatrix}
0 \\
1 \\
v_i^{\ast^2} \\
\ell_i v_i^\ast
\end{bmatrix} \right\}
\end{align*}
From above, we know there exists $\bar{u}_i$ that satisfies the above vector equality with at most $r + 2$ non-binary entries. We can compute this $\bar{u}$ via the linear program in~\eqref{eq:lp-sf-pen-q}. We then primalize precisely as before to arrive at 
the chain of inequalities
\begin{align*}
p^\ast_\con(k + r + 2)   \leq  OPT \leq p^{\ast\ast}_\con(k)\leq p^\ast_\con(k)
\end{align*}
which is the desired result.
\end{proof}

\subsection{$\ell_2$ Constrained, $\ell_0$ Penalized Optimization}
The analysis for the penalized case is very similar to that of Section \ref{sec:l0_pen}. 
\subsubsection{Dual}
The penalized problem is equivalent to
\begin{align}
    p^\ast_\pen (\lambda) = \min_{v, u\in \{0,1\}^m} f(XD(u) v) + \lambda \ones^\top u \; :\; v^\top D(u) v\leq \gamma
\end{align}
Using the fenchel conjugate of $f$, introducing a dual variable $\eta$ for the $\ell_2$ constrain, using weak duality and computing the minimization over $v$ we have
\begin{align*}
    d^\ast_\pen (\lambda) = \max_{z, \eta \geq 0} -f^\ast(z) - \dfrac{\eta \gamma}{2} + \min_{u\in\{0,1\}^m} - \dfrac{1}{2\eta} z^\top XD(u) D(u)^\dagger D(u) X^\top z + \lambda \ones^\top u
\end{align*}
Using the fact
\begin{align*}
    \min_{u\in\{0,1\}^m} - \dfrac{1}{2\eta} z^\top XD(u) D(u)^\dagger D(u) X^\top z + \lambda \ones^\top u = \sum_{i=1}^m \min\Big(0, \lambda - \tfrac{1}{2\eta} (X^\top z)_i^2\Big)
\end{align*}
the dual problem becomes
\begin{align*}
    d^\ast_\pen(\lambda) = \max_z -f^\ast(z)- \dfrac{\eta \gamma}{2} +\sum_{i=1}^m \min\Big(0, \lambda - \tfrac{1}{2\eta} (X^\top z)_i^2\Big)
\end{align*}
with $d^\ast_\pen(\lambda) \leq p^\ast_\pen(\lambda)$. The term $\tfrac{1}{2\eta} (X^\top z)_i^2$ is jointly convex since it can be recast as a second order cone constraint using the fact that $z_i^2 /\eta \leq t \Longleftrightarrow \Big\|\begin{bmatrix}
z_i \\
t - \eta 
\end{bmatrix}\Big\|_2 \leq \tfrac{1}{2}(t + \eta)$.
\subsubsection{Bidual}
Rewriting the second term of our objective in variational form we have
\begin{align*}
    p^{\ast \ast}_\pen (\lambda) = d^\ast(\lambda) = \max_{z, \eta \geq 0} \min_{u \in [0,1]^m} -f^\ast(z) - \dfrac{\eta \gamma}{2} - \dfrac{1}{2\gamma} z^\top XD(u) D(u)^\dagger D(u) X^\top z + \lambda \ones^\top u 
\end{align*}
Performing the same analysis as for the constrained case, we have that 
\begin{align}
    p^{\ast \ast}_\pen (\lambda) = \min_{v,u \in [0,1]^m} f(XD(u) v) + \lambda \ones^\top u \; : \; v^\top D(u) v \leq \gamma 
\end{align}
which can be recast as a convex program (c.f. Section \ref{sec:l0_cons_bd}).

\begin{corollary}\label{prop:pen-gap-l2}
Suppose $X = U_r \Sigma_r V_r^\top$ is a compact, rank-$r$ SVD decomposition of $X$. From a solution $(v^\ast, u^\ast)$ of~\eqref{eq:bidual-pen} with objective value $t^\ast$, with probability one, we can construct a point with objective value OPT satisfying
\BEQ\label{eq:pen-gap-2}
    p^{\ast\ast}_\pen(\lambda) \leq p_\pen (\lambda)  \leq OPT \leq p^{\ast\ast}_\pen (\lambda) + \lambda (r + 1) \tag{Gap-Bound-Pen-l2}
\EEQ
by solving a linear program written
\BEQ\label{eq:lp-sf-pen-q2}
\BA{ll}
\mbox{minimize} & c^\top u \\
\text{s.t} \;\;&f^{\ast \ast} (U_r z^\ast) + \lambda u_i = t^\ast \\
& \sum_{i=1}^m u_i  v_i^{\ast^2} \leq \gamma \\
& \sum_{i=1}^m u_i \ell_i v_i^\ast = z^\ast \\
& u \in [0,1]^m
\EA\EEQ
in the variable $u\in\reals^m$ with $z^\ast=\Sigma_r V_r^\top D(u^\ast) v^\ast$.
\end{corollary}
\begin{proof}
The primalization procedure is analogous to the constrained case with the only difference being the linear program becoming~\eqref{eq:lp-sf-pen-q2}. Performing the same analysis as for the penalized case, we have the chain of inequalities in~\eqref{eq:pen-gap-2}.
\end{proof}

\section{Tighter Bounds using the Numerical Rank}
\label{sec:general}

The duality gap bounds detailed above depend on $r$, the rank of the matrix $X$, which is an unstable quantity. In other words, a very marginal change in $X$ can have a significant impact on the quality of the bounds. In what follows, we will see how to improve these bounds when the matrix $X$ is approximately low rank. This will allow us to bound the duality gap using the (stable) numerical rank of $X$.

Starting from the $\ell_2-\ell_0$ constrained formulation, we formulate a perturbed version 
\BEQ\label{eq:pert-primal}
\BA{rll}
p^\ast_{\con}(k,X,\delta) =& \min &f(z; y) \\
& s.t. & Xw=z + \delta,\\
&  & \|w\|_0 \leq k \\
&  & \|w\|_2^2 \leq \gamma
\EA\EEQ
in the variables $w\in\reals^m$ and $z\in\reals^n$, where $\delta\in\reals^n$ is a perturbation parameter. Let
\[
X= X_r + \Delta X, \quad \Rank{X_r}=r,
\]
be a decomposition of the matrix $X$. For notational convenience, we set $p^\ast_\con(k,X) = p^\ast_\con(k,X,0)$. We have the following result.

\begin{proposition}\label{prop:soft}
Let $w^\star_r$ be the optimal solution of $p^\ast_{\con}(k,X_r,0)$ and $\nu_r^*$ the dual optimal variable corresponding to the equality constraint, and write $(w^\star,\nu^\star)$ the corresponding solutions for $p^\ast_{\con}(k,X,0)$. We have
\BEQ
p^\ast_{\con}(k,X,0) - \nu^{*T}\Delta Xw^\star_r \leq p^\ast_{\con}(k,X_r,0) \leq p^\ast_{\con}(k,X,0) - \nu^{*T}_r\Delta Xw^\star.
\EEQ
and the exact same bound when we start with the $\ell_2$ constrained, $\ell_0$ penalized formulation.
\end{proposition}
\begin{proof}
Suppose $w^\star_r$ is an optimal solution of problem $p^\ast_{\con}(k,X_r,0)$, then $w^\star_r$ is also a feasible point of problem $p^\ast_{\con}(k,X,\Delta Xw^\star_r)$ because
\[
(X_r + \Delta X) w^\star_r = z + \Delta Xw^\star_r
\]
by construction. Since the two problems share the same objective function, this means $p^\ast_{\con}(k,X_r,\Delta Xw^\star_r) \leq p^\ast_{\con}(k,X_r,0)$. Now, weak duality yields 
\[
p^\ast_{\con}(k,X,\Delta Xw^\star_r) \geq p^\ast_{\con}(k,X,0) - \nu^{\star T} \Delta Xw^\star_r
\]
and
\[
p^\ast_{\con}(k,X_r,0) \leq p^\ast_{\con}(k,X_r,-\Delta Xw^\star) - \nu^{\star T}_r \Delta Xw^\star.
\]
We conclude using as above the fact that if $w^\star$ is an optimal solution of problem $p^\ast_{\con}(k,X,0)$, then $w^\star$ is also a feasible point of problem $p^\ast_{\con}(k,X_r,\Delta Xw^\star)$ because $X_r w^\star = z - \Delta Xw^\star$ which yields $p^\ast_{\con}(k,X_r,-\Delta Xw^\star) \leq p^\ast_{\con}(k,X,0)$ and the desired result. 
In the proof we only used weak duality and the equality constraint in $p^\ast_\con(k,X,\delta)$ to arrive at the result. Consequently, the exact same proof and bounds hold for $p^\ast_\pen(\lambda,X,\delta)$.
\end{proof}

We are now ready to combined the bound in Proposition \ref{prop:soft} with the bounds derived in Section \ref{sec:relaxation_cons}.

\begin{proposition}\label{prop:fullbound}
Let $w^\star_r$ be the optimal solution of $p^\ast_{\con}(k,X_r)$ and $\nu_r^*$ the dual optimal variable corresponding to the equality constraint, and write $(w^\star,\nu^\star)$ the corresponding solutions for $p^\ast_{\con}(k,X)$. Futhermore, let $\zeta_r = \sqrt{\gamma}\|\Delta X^\top \nu^\ast_r\|_2$ and $\zeta = \sqrt{\gamma}\|\Delta X^\top \nu^\ast\|_2$. We have
\BEQ
-\zeta_r - \zeta + p^{\ast \ast}_\con (k + r + 2, X_r) \leq  OPT \leq p_\con (k,X) - \zeta \leq p^{\ast \ast}_\con (k, X_r)
\EEQ
Similarly, for $p^\ast_{\pen}(\lambda,X_r)$ we have
\BEQ
-\zeta_r -\zeta + p^{\ast \ast}_{\pen}(\lambda, X_r) \leq p_{\pen}(\lambda, X) - \zeta \leq OPT \leq p^{\ast \ast}_{\pen}(\lambda, X_r) + \lambda(r+1) 
\EEQ
\end{proposition}
\begin{proof}
Starting from $p_{\con}(k,X) = p_{\con}(k,X) - p_{\con}(k,X_r) + p_{\con}(k,X_r)$, upper and lower bounding $p_{\con}(k,X) - p_{\con}(k,X_r)$ using Proposition \eqref{prop:soft} and the Cauchy-Schwarz inequality, and using the bounds derived in Section \ref{sec:relaxation_cons} the result follows. The proof for the penalized case is identical.
\end{proof}
\section{Experiments}
\label{sec:expt}

\subsection{Experiment 1: Duality gap bounds}
In this experiment, we generate synthetic data to illustrate the duality gap bounds derived in Sections \ref{sec:l0_cons} and \ref{sec:l0_pen}. We plot these  bounds for the $f(Xw;y) = \tfrac{1}{2n} \|Xw-y\|_2^2 $ (linear regression) and $f(Xw;y) = \tfrac{1}{n}\sum_{i=1}^n \log (1 + \exp(-y_i(x_i^\top w)))$ (logistic regression). Note that both functions are convex and closed; hence $f^{\ast \ast} = f$ and $\rho(f) = 0$. Specifically, we generate samples $X \in \mathbb{R}^{1000 \times 100}$ with $\text{rank}(X) = 10$ by first generating $X_{ij} \sim \mathcal{N}(0,1)$ and then taking a rank-10 SVD. We generate $\beta \in \mathbb{R}^{100}$ with $\beta_i \sim \mathcal{N}(0,25)$ and $\|\beta\|_0 = 10$. In the case of $\ell_2$ loss, we set $y = X\beta + \epsilon$ and for the logistic loss we set $y = 2\text{Round}(\text{Sigmoid}(X\beta + \epsilon)) - 1 \in \{-1,1\}^n$ where $\epsilon_i \sim \mathcal{N}(0,1)$. For both models we add a ridge penalty $\tfrac{\gamma}{2}\|w\|_2^2$ with $\gamma = 0.01$. For the regression task, we use a $\ell_0$-penalty while for the classification task, we use a $\ell_0$-constraint. Figure \ref{fig:l2pen} shows the primalized optimal values as well as the upper and lower bounds derived earlier. When running the primalization procedure, we pick twenty random linear objectives and show the standard deviation in the value OPT.

\begin{figure}[ht!]
  \centering
  \begin{minipage}[b]{0.49\textwidth}
    \includegraphics[width=\textwidth]{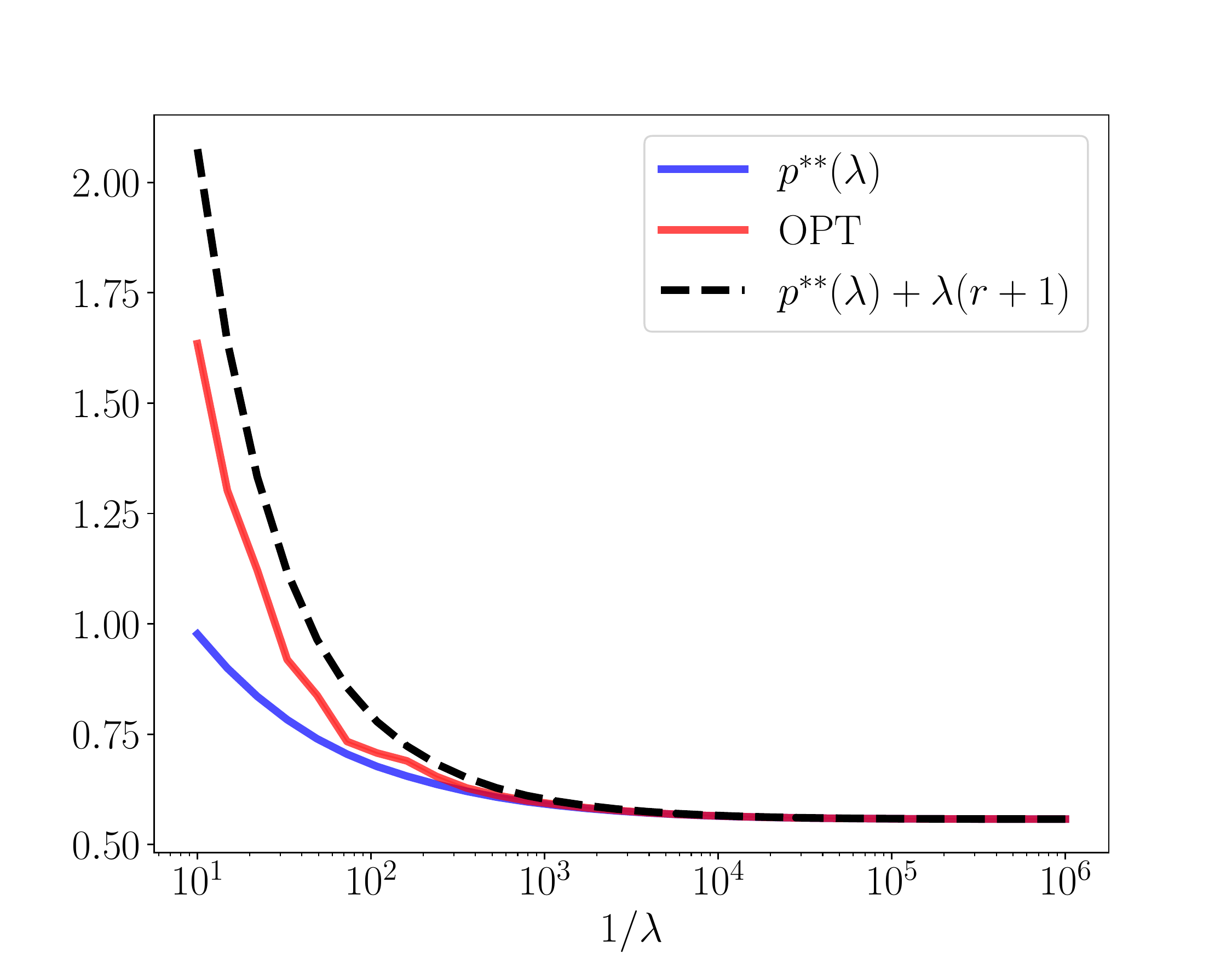}
  \end{minipage}
  \hfill
  \begin{minipage}[b]{0.49\textwidth}
    \includegraphics[width=\textwidth]{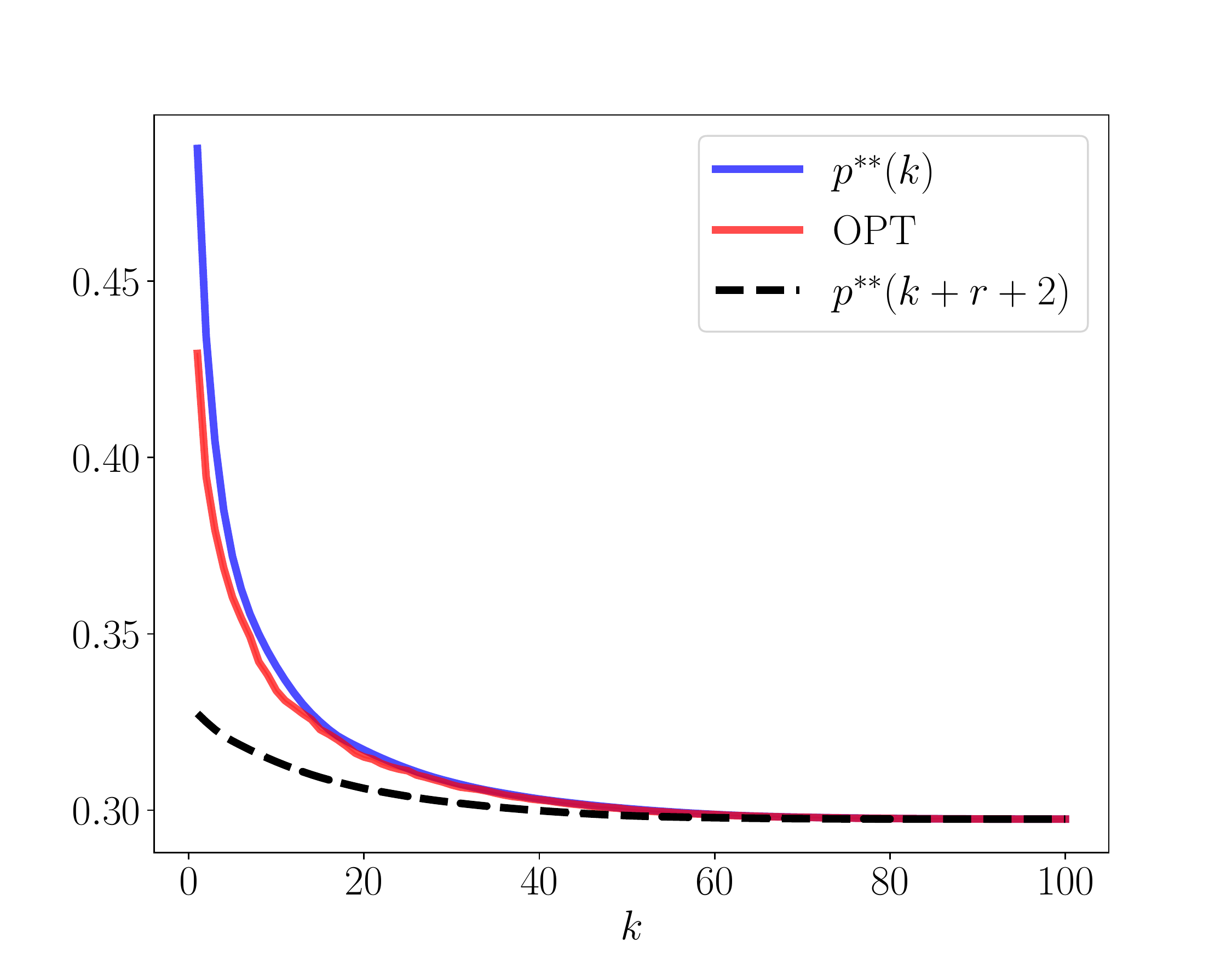}
  \end{minipage}
  \caption{ \textbf{Experiment 1} \textit{(Left)} Linear regression with a $\ell_0$-penalty. \textit{(Right)} Logistic regression with a $\ell_0$-constraint.
  }
  \label{fig:l2pen}
\end{figure}

Note that there are no error bars around OPT despite having solved the primalization linear program with 20 different random linear objectives for each value of the regularization parameter $(\lambda$ or $k)$. This strongly indicates that our feasible set for the linear program is actually a singleton (which was verified by changing the linear objective to arbitrary convex objectives and noting the $\arg\min$ was identical each time). In this case, the solution is identical to the solution that can be inferred from the bidual (since we know the linear program is feasible since the solution of the bidual satisfies the constraints). As a result, primalization simply reduces to rounding the bidual solution to make it primal feasible. Furthermore, note that in the left plot of Figure \ref{fig:l2pen}, we know that the true value $p^\ast_\pen(\lambda)$ must lie somewhere between $OPT$ (red line) and $p^{**}$ (blue line) and that this gap decreases as $\lambda$ decreases. This is also apparent in the right plot of Figure \ref{fig:l2pen} as the marginal importance of the features decreases as $k$ increases. 

\subsection{Experiment 2: Numerical rank bounds}
In this experiment, we plot the bounds outlined by Proposition \ref{prop:fullbound} that combine Shapley Folkman with numerical rank bounds. We generate $X \in \mathbb{R}^{1000 \times 100}$ with bell shaped singular values using the \texttt{make\_low\_rank\_matrix} function in sklearn \cite{pedregosa2011scikit} to get a numerical rank of 10. We then generate $\beta$ and $y$ as in Experiment 1 for the $\ell_2$ loss. As was used to derive the numerical rank bounds, we use a constraint $\|w\|_2^2 \leq \gamma$ with $\gamma = 30$ instead of a ridge penalty. We consider the $\ell_0$-penalized case and fix three values of $\lambda: 10^{-4}, 10^{-3}, 10^{-2}$. In Figure \ref{fig:l2con}, we show how the bounds change as we vary the rank of our approximation $X_r$ from 1 to 100. While running the primalization procedure, we pick a random linear objective 20 times and show the standard deviation in the value of OPT.

From Proposition \ref{prop:fullbound}, we know that $p_\pen(\lambda, X)$ lies between the red and blue lines. For small values of $\lambda$ (e.g. $\lambda = 10^{-4}$) we see that as the rank increases this gap is essentially zero. This means in the case of $\lambda = 10^{-4}$ taking a rank 20 approximation of the data matrix or a rank 100 matrix and doing the procedure highlighted in Section \ref{sec:relaxation_cons} results in two different solutions that are both essentially optimal. Both plots at the bottom of Figure~\ref{fig:l2con} highlight a trade-off in choosing the numerical rank, a lower rank improves the duality gap while it coarsens the objective function approximation, and vice-versa, this is further illustrated in the experiment below.

\begin{figure}[h!]
    \centering
    \includegraphics[width=\textwidth]{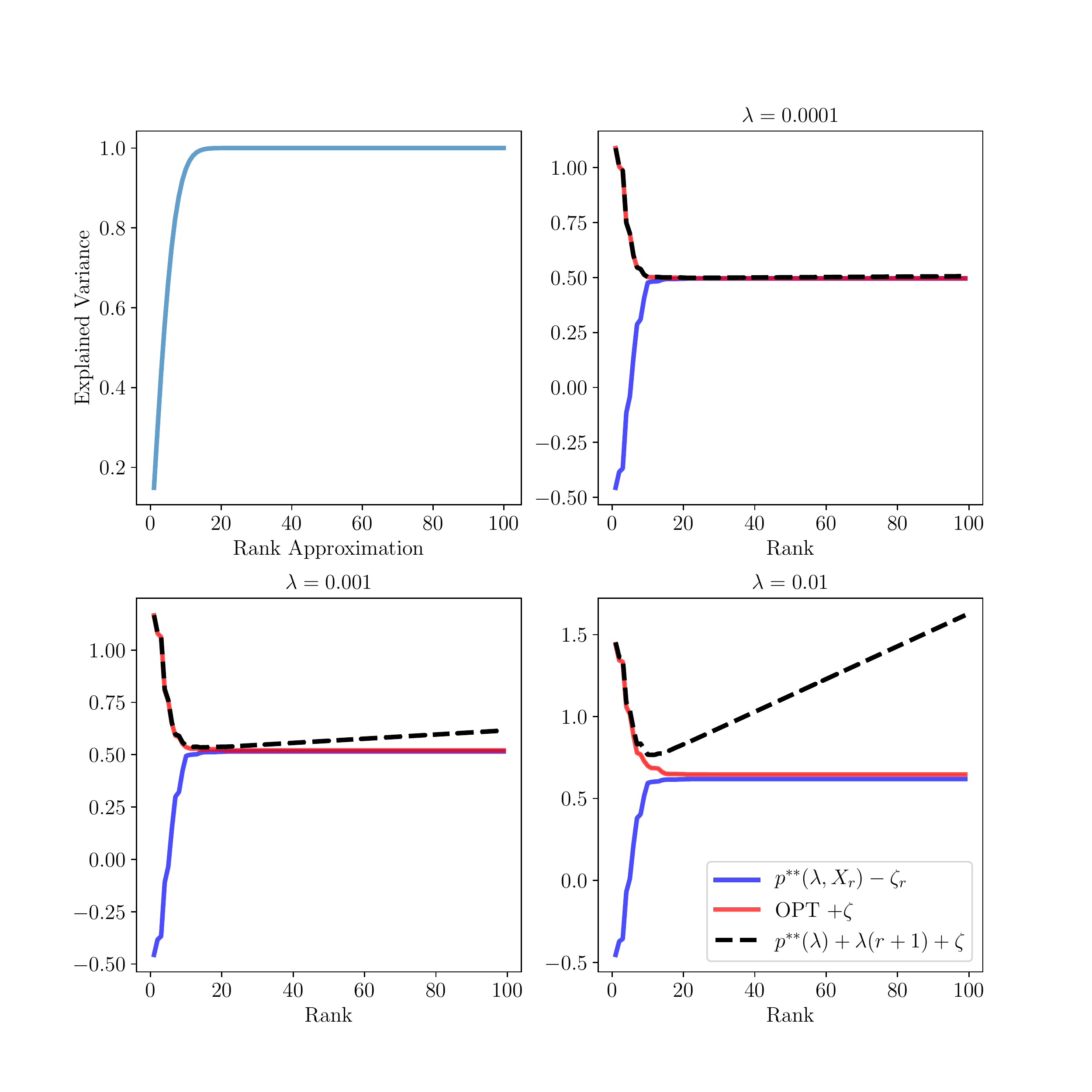}
    \caption{\textbf{Experiment 2} \textit{(Top Left)} Explained variance plot of successive rank-$r$ SVD approximations of $X$. Note that rank 10 explains most of the variance. \textit{(Top Right, Bottom)} Plot of upper and lower bounds on $p_{\pen}(\lambda, X)$ for various values of $\lambda$ as the rank-$r$ approximation of $X$ changes.}
    \label{fig:l2con}
\end{figure}

\subsection{Experiment 3: Numerical Rank Bounds on Natural Data Sets} In this experiment, we generate the same plot as in Experiment 2 but now with real data. Specifically, we use the Leukemia data \cite{dettling2004bagboosting} with $n=72$ binary responses and $m=3751$ features. We scale the data matrix $X$ and then plot the difference between the upper and lower bounds (duality gap) and the difference between the primalized upper bound and lower bounds (primalized gap) in Figure \ref{fig:l2con-leukemia} under a logistic loss with $\lambda = 0.1$ and $\gamma = 50$.

\begin{figure}[h!]
    \centering
    \includegraphics[width=\textwidth]{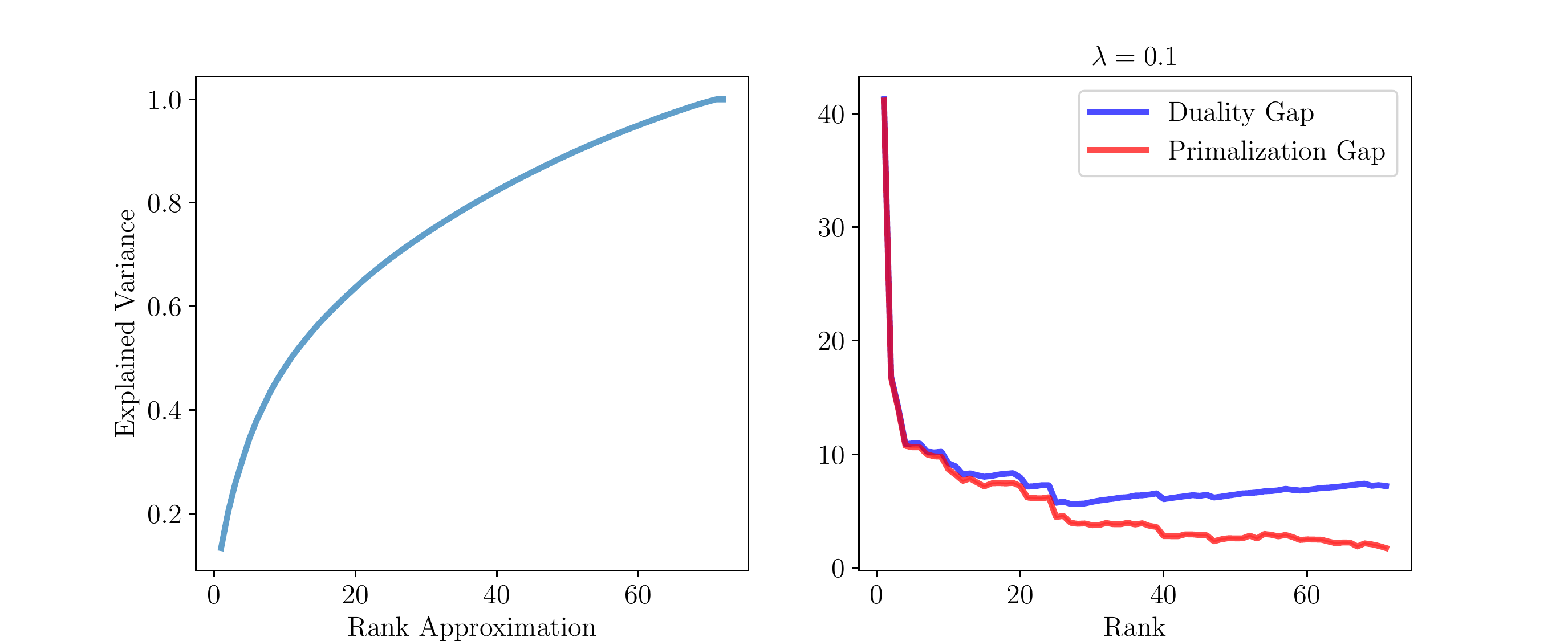}
    \caption{\textbf{Experiment 3} \textit{(Left)} Explained variance plot of successive rank-$r$ SVD approximations of $X$. \textit{(Right)} Difference between upper bounds and primalized value with lower bound as in Proposition \ref{prop:fullbound}.}
    \label{fig:l2con-leukemia}
\end{figure}



 
\bibliography{ref,MainPerso}
\bibliographystyle{unsrt}

\onecolumn

\appendix



\end{document}